\documentclass[reqno]{amsart}
\usepackage{amssymb,amsthm,amsfonts,amstext}
\usepackage{amsmath}
\usepackage{newcent}       
\usepackage{hyperref}
\texorpdfstring

\usepackage[left=1.5in,right=1.5in]{geometry}

\usepackage{url}
\usepackage{bbm}

\makeatletter \@addtoreset{equation}{section} \makeatother

\renewcommand\thefigure{\thesection.\@arabic\c@figure}
\renewcommand\thetable{\thesection.\@arabic\c@table}
\newtheorem{theorem}{Theorem}[section]
\newtheorem{lemma}[theorem]{Lemma}
\newtheorem{proposition}[theorem]{Proposition}
\newtheorem{corollary}[theorem]{Corollary}

\theoremstyle{definition}
\newtheorem{remark}[theorem]{Remark}

\newcommand{\mc}[1]{{\mathcal #1}}

\newcommand{\bb}[1]{{\mathbb #1}}

\newcommand{\<}{\langle}
\renewcommand{\>}{\rangle}

\DeclareMathOperator{\supp}{supp}

\DeclareMathOperator{\Var}{Var}

\title[Additive functionals of particle systems]{Scaling limits of additive functionals of interacting particle systems}

\author{Patr\'{\i}cia Gon\c{c}alves}
\address{CMAT, Centro de Matem\'atica da Universidade do Minho, Campus de Gualtar, 4710-057 Braga, Portugal}
\email{patg@math.uminho.pt and Patricia.Goncalves@cims.nyu.edu}

\author{Milton Jara}

\address{{IMPA, Instituto Nacional de Matem\'atica Pura e Aplicada\\
Estrada Dona Castorina 110\\
Jardim Bot\^anico\\
22460-320 Rio de Janeiro-RJ\\
Brazil} and {Ceremade, UMR CNRS 7534,
    Universit\'e de Paris IX - Dauphine,
    Place du Mar\'echal De Lattre De Tassigny
    75775 Paris Cedex 16 - France}
\newline
e-mail: \rm \texttt{mjara@impa.br}}


\begin{document}

\begin{abstract}
Using the renormalization method introduced in \cite{GJ}, we prove what we call the {\em local} Boltzmann-Gibbs principle for conservative, stationary interacting particle systems in dimension $d=1$. As applications of this result, we obtain various scaling limits of additive functionals of  particle systems, like the occupation time of a given site or extensive additive fields of the dynamics. As a by-product of these results, we also construct a novel process, related to the stationary solution of the stochastic Burgers equation.
\end{abstract}

\maketitle

\section{Introduction}

A classical problem in the theory of Markov processes is the study of {\em additive functionals} of the trajectory of the process. More precisely, for a given Markov process $\{\eta_t; t \geq 0\}$ in a state space $\Omega$ and a suitable function $f: \Omega \to \bb R$, we are interested on the long-time behavior of
\[
 \Gamma_t(f) = \int_0^t f(\eta_s)ds
\]
and the possible scaling limits of $\Gamma_t$. Among the vast literature on the subject, we point out the seminal work of Kipnis and Varadhan \cite{KV}, from which we build up the results of this article. Under the assumptions of reversibility and stationarity of $\{\eta_t;t \geq 0\}$, Kipnis-Varadhan's theorem gives a complete characterization of the functions $f$ for which $\Gamma_t$ has a Brownian motion as scaling limit. It is known that for interacting particle systems in low dimensions, for very simple functions $f$ the process $\Gamma_t$ does not have a Brownian motion as scaling limit, either because reversibility does not hold (see \cite{CG} for instance) or because $f$ does not verify the assumptions of Kipnis-Varadhan's theorem (see \cite{Kip} for instance). In these works the authors considered the case on which the function $f$ is the number of particles at the origin. Since in those models the particle system admits at most one particle per site, the process $\Gamma_t(f)$ is called the {\em occupation time} of the origin.

In this work we address the question of obtaining the scaling limit of the process $\{\Gamma_t(f); t \geq 0\}$ in the case of one-dimensional, conservative, stationary interacting particle systems. It is known \cite{SX} that in dimension $d \geq 3$ and for any local function $f$, the scaling limit of $\Gamma_t(f)$ is given by a Brownian motion. In a series of works (\cite{Set1,Set2,Set3}), the author explores various cases on which the scaling limit of occupation times is not Brownian, obtaining in some cases a functional central limit theorem and in other cases bounds on the variance close to the conjectured ones. In \cite{Ber} and \cite{Ber2} the author obtain some non-matching upper and lower bounds for the occupation time of asymmetric particle systems, which precludes a Brownian scaling limit. In \cite{JQS}, the authors obtained the scaling limit of additive functionals of a zero-range process, extending the results of \cite{Kip} and \cite{Set1} to the zero-range process. The results of \cite{Set1} and of \cite{JQS} are based on the martingale method, which consists in writing down the process $\Gamma_t(f)$ as a martingale plus a vanishing term. Scaling limits of the process $\Gamma_t(f)$ then follow from standard results about martingale scaling limits.

Our approach is essentially different. In \cite{GJ} an important technical novelty, the so-called {\em second-order Boltzmann-Gibbs principle} was introduced. The idea was to extend the one-block and the two-blocks setup introduced in \cite{GPV} to the fluctuation level. Following the proof in \cite{GJ}, we obtain here what we call the {\em local} Boltzmann-Gibbs principle, which is of both practical and theoretical interest. The local Boltzmann-Gibbs principle states that the process $\Gamma_t(f)$ is well approximated by the density of particles on a box of size $\varepsilon \sqrt t$ around the origin, integrated up to time $t$. In particular, if a scaling limit in the diffusive scaling is available for the density of particles, a scaling limit for $\Gamma_t(f)$ can be extracted from it through approximation arguments.

Let $n$ be a scaling parameter. In order to obtain the scaling limit of the density of particles, a diffusive scaling must be introduced. This is done rescaling space by $1/n$ and time by $n^2$.
For diffusive, conservative particle systems, the scaling limit of the density of particles is given by an infinite-dimensional Ornstein-Uhlenbeck process.
Starting from this result and using an approximation procedure, we can prove that for diffusive systems, the scaling limit of the process $\Gamma_t(f)$ is given by a fractional Brownian motion of Hurst index $H=3/4$. In particular, we solve one of the open questions in \cite{Set1}, namely the scaling limit of $\Gamma_t$ for the mean-zero exclusion process in dimension $d=1$. The scaling limit of the density of particles is also understood in other two situations.  When the motion of the particles is weakly asymmetric, namely when the particles have a drift of order $O(1/n)$ the scaling limit of the density of particle is given by a infinite-dimensional Ornstein-Uhlenbeck process with drift. In that case we prove that the scaling limit of $\Gamma_t(f)$ is given by a Gaussian process with stationary increments, with some prescribed variance.

A case which has received a lot of attention recently (see \cite{ACQ,GJ} and references therein) is the so-called {\em KPZ scaling}. In this case, particles have a drift of order $O(1/\sqrt n)$ and density fluctuations are observed along characteristic lines of the system. The fluctuations of the density are governed by solutions of a stochastic Burgers equation \cite{BG} which is formally the derivative of the celebrated {\em KPZ equation}, which arose as a continuum model for the stochastic growth of interfaces. We prove that there exists a process $\{\mc Z_t; t \geq 0\}$ (for which we basically know nothing) which can be constructed for any energy solution of the KPZ equation (as introduced in \cite{GJ}) such that the scaling limit of $\Gamma_t(f)$ is given by $c(f)\mc Z_t$, where $c(f)$ is an explicit constant depending on $f$. It seems that there is no previous mention of this process neither in the physics nor in the mathematics literature. The description of the limiting process $\mc Z_t$ remains an open problem.

In \cite{Ass} a quadratic field associated to the Ornstein-Uhlenbeck process was introduced. It was also proved that the scaling limit of certain extensive additive functionals of the symmetric simple exclusion process is given by this quadratic field. In \cite{GJ} it was proved that the so-called energy condition implies the existence of the quadratic field associated to a given stochastic field. In this article we prove that scaling limits of extensive additive fields of one-dimensional, conservative particle systems are given either by linear functionals of the limiting density field or by the quadratic fields associated to the limiting density field, therefore solving an open problem posed by the author in \cite{Ass}.

In order to have a good compromise between generality and simplicity, we study in this paper a particular one-dimensional system, known in the literature as a {\em lattice gas} dynamics at infinite temperature or also as the speed-change exclusion process. From the technical point of view, this system is non-gradient. Fortunately, the non-gradient method is only needed to obtain the scaling limit of the density of particles, which has already been done in \cite{Cha2} and \cite{Sel}. Our method requires two main assumptions on the dynamics, namely a sharp estimate on the spectral gap of the dynamics restricted to finite boxes and a second-order equivalence of ensembles for the invariant measures of the system. These two assumptions are the same needed in \cite{Cha2}, \cite{Sel} in order to derive the scaling limit of the density of particles, so in that sense our result is the best possible.

The paper is organized in the following way. In Section \ref{s1} we fix the notation used throughout the paper and we state our main results about scaling limits of occupation times and extensive additive functionals. In Section \ref{s2} we review the spectral gap inequality, the equivalence of ensembles and Kipnis-Varadhan's inequality, which will be the building blocks of the proofs in Section \ref{s3}. In Section \ref{s3} we state and prove the local Boltzmann-Gibbs principle. The proof follows the renormalization procedure introduced in \cite{GJ}. In Section \ref{s4} we prove the theorems stated in Sections \ref{s1.3} and \ref{s1.4}. And in Section \ref{s5} we roughly explain how to extend the main results to the context of mean zero and weakly asymmetric exclusion processes.

\section{Notations and results}
\label{s1}
\subsection{Lattice gas dynamics}
\label{s1.1}
Let $\Omega = \{0,1\}^{\bb Z}$ be the state space of a Markov process which we define below. We denote by $\eta = \{ \eta(x); x \in \bb Z\}$ the elements of $\Omega$. We say that a function $f : \Omega \to \bb R$ is {\em local} if there exists a non-negative integer $R$ such that $f(\eta) = f(\xi)$ whenever $\eta(x) =\xi(x)$ for every $x \in \bb Z$ with $|x|\leq R$. For each $x \in \bb Z$ we denote by $\tau_x$ the usual translation (or shift) in $\bb Z$ of magnitude $x$. In other words, $\tau_x: \Omega \to \Omega$ is defined by $\tau_x \eta(z) = \eta(z-x)$. For a function $f: \Omega \to \bb R$, $\tau_x f$ is defined in an analogous way: $\tau_x f(\eta) = f(\tau_x \eta)$ for any $\eta \in \Omega$. Let $r: \Omega \to \bb R$ be a local function satisfying the following conditions:
\begin{itemize}
 \item[i)]{\em Ellipticity}. There is a constant $\varepsilon_0 >0$ such that $\varepsilon_0 \leq r(\eta) \leq \varepsilon_0^{-1}$ for any $\eta \in \Omega$.
 \item[ii)]{\em Reversibility}. For any $\eta,\xi \in \Omega$ such that $\eta(x)= \xi(x)$ for every $x \neq 0,1$, we have $r(\eta) = r(\xi)$.
\end{itemize}

Let us define $r_x =\tau_x r$. For $x,y \in \bb Z$ and $\eta \in \Omega$ we define $\eta^{x,y} \in \Omega$ as
\[
 \eta^{x,y}(z)=
\begin{cases}
 \eta(y), & z=x\\
\eta(x), & z=y\\
\eta(z), &z \neq x,y.
\end{cases}
\]
For $x,y \in \bb Z$ and $f: \Omega \to \bb R$ we define $\nabla_{x,y} f :\Omega \to \bb R$ as $\nabla_{x,y} f(\eta) = f(\eta^{x,y}) -f(\eta)$.
The {\em lattice gas} with interaction rate $r$ is defined as the Markov process $\{\eta_t; t \geq 0\}$ with state space $\Omega$ and generated by the operator $L$, whose action over local functions $f:\Omega \to \bb R$ is given by
\[
 L f = \sum_{x \in \bb Z} r_x \nabla_{x,x+1} f.
\]
Notice that, since $f$ is local, only a finite number of terms in the previous sum are different from zero. The process $\{\eta_t; t \geq 0\}$ is well defined, thanks to the ellipticity and translation invariance of the family of rates $\{r_x; x \in \bb Z\}$ (see Section 1.3 of \cite{Lig}).

Notice that particles are neither created nor annihilated by the dynamics of $\eta_t$. The process $\{\eta_t; t \geq 0\}$ has a family of invariant measures parameterized by the density of particles. For each $\rho \in [0,1]$, let $\nu_\rho$ denote the product Bernoulli measure in $\Omega$ of density $\rho$, that is, $\nu_\rho$ is the only probability measure in $\Omega$ such that
\[
 \nu_\rho\big\{ \eta(x) = 1 \text{ for any } x \in A\big\} =\rho^{|A|}
\]
for every finite set $A \subseteq \bb Z$. Here $|A|$ denotes the cardinality of the set $A$. Thanks to the reversibility condition, the measures $\{\nu_\rho;  \rho \in [0,1]\}$ are invariant and reversible under the evolution of $\{\eta_t; t \geq 0\}$. Thanks to the ellipticity condition, these measures are ergodic as well. Notice that $\int \eta(x) d\nu_\rho = \rho$ for any $x \in \bb Z$ and any $\rho \in [0,1]$, justifying the denomination ``density of particles'' for the parameter $\rho$.

In this article we are interested in the evolution of some observables of the process $\{\eta_t;t \geq 0\}$ starting from an equilibrium measure $\nu_\rho$. Therefore, from now on we fix $\rho \in (0,1)$ and we will always consider the process $\{\eta_t; t \geq 0\}$ with initial distribution $\nu_\rho$. Notice that when $\rho =0$ or $\rho =1$ the evolution of $\eta_t$ is trivial.

Let $T>0$ be fixed from now on and up to the end of the article. This convention will simplify the notation and the exposition. For a given Polish space $X$ we denote by $\mc C([0,T],X)$ the space of continuous trajectories from $[0,T]$ to $X$, and we denote by $\mc D([0,T],X)$ the space of c\`adl\`ag trajectories from $[0,T]$ to $X$. We denote by $\bb P_\rho$ the distribution in $\mc D([0,T],X)$ of the process $\{\eta_t; t \in [0,T]\}$ with initial distribution $\nu_\rho$, and we denote by $\bb E_\rho$ the expectation with respect to $\bb P_\rho$.

\subsection{Mean-zero exclusion process}
\label{s1.2}
One of the most popular non-reversible, diffusive interacting particle systems is the mean zero exclusion process, which we define as follows.
Let $p: \bb Z \setminus \{0\} \to [0,1]$ be a probability measure. For notational convenience, we define $p(0)=0$. We assume that $p(\cdot)$ has finite range, that is, there exists $M >0$ such that $p(z)=0$ whenever $|z| > M$ and that $p(\cdot)$ is irreducible, that is, $\bb Z = \mathrm{span}\{ z \in \bb Z; p(z)>0\}$. We also assume that $p(\cdot)$ has {\em mean zero}, namely,
\[
 \sum_{z \in \bb Z} z p(z) =0.
\]
A simple example of a probability $p(\cdot)$ satisfying these properties is given by $p(1)= 2/3$, $p(-2)=1/3$, $p(z)=0$ if $z \neq -2,1$. The mean-zero exclusion process associated to $p(\cdot)$ is the Markov process $\{\eta_t^{\mathrm{ex}}; t \geq 0\}$ generated by the operator $L_{\mathrm{ex}}$ whose action over local functions $f:\Omega \to \bb R$ is given by
\[
 L_{\mathrm{ex}} f =\sum_{x,y \in \bb Z} p(y-x) r_{x,y} \nabla_{x,y} f,
\]
where we have written $r_{x,y}(\eta) = \eta(x)(1-\eta(y))$.
Like in the case of the lattice gas described in Section \ref{s1.1}, the existence of this process is shown in \cite{Lig}. The measures $\{\nu_\rho; \rho \in [0,1]\}$ are invariant and ergodic under the evolution of $\eta_t^{\mathrm{ex}}$ but they are not necessarily reversible. In fact, a straightforward computation shows that the adjoint of $L_{\mathrm{ex}}$ in $ L^2(\nu_\rho)$ is the generator of the mean-zero exclusion process associated to the probability $p^*(z)= p(-z)$. Therefore, $\nu_\rho$ is reversible under the evolution of $\eta_t^{\mathrm{ex}}$ if and only if $p(\cdot)$ is symmetric. We point out that the results of this article do not rely on the reversibility of the invariant measure $\nu_\rho$.

\subsection{Infinite-dimensional Ornstein-Uhlenbeck process}
\label{s1.3}
Let $\mc S(\bb R)$ be the Schwartz space of test functions and let $\mc S'(\bb R)$ be the topological dual of $\mc S(\bb R)$, corresponding to the space of tempered distributions in $\bb R$. For $u, v \in L^2(\bb R)$, let $\<u,v\> = \int u(x) v(x) dx$ be the inner product in $ L^2(\bb R)$. We write $\|u\| = \<u,u\>^{1/2}$ for the $L^2(\bb R)$-norm of $u$.

We say that a process $\{ \mc M_t; t \in [0,T]\}$ with trajectories in $\mc C([0,T],\mc S'(\bb R))$ is a $\mc S'(\bb R)$-valued Brownian motion if for any $u \in \mc S(\bb R)$, the process $\{\mc M_t(u); t \in [0,T]\}$ is a Brownian motion of mean zero and variance $\|u\|^2$.

Let $D, \sigma>0$ be fixed. We say that a process $\{\mc Y_t; t \in [0,T]\}$ is a solution of the infinite-dimensional Ornstein-Uhlenbeck equation
\begin{equation}
\label{ec1}
 d \mc Y_t = D \Delta \mc Y_t dt + \sigma \nabla d \mc M_t
\end{equation}
if for any smooth trajectory $t \mapsto u_t$ from $[0,T]$ to $\mc S(\bb R)$, the process
\[
 \mc Y_t(u_t) -\mc Y_0(u_0) - \int_0^t \mc Y_s((\partial_s + D \Delta)u_s) ds
\]
is a martingale of quadratic variation $\sigma^2 \int_0^t \|\nabla u_s\|^2 ds$. According to \cite{HS}, the Cauchy problem for \eqref{ec1} is well-posed for any distribution $\mc Y_0$ in $\mc S'(\bb R)$.

We say that a $\mc S'(\bb R)$-valued random variable $\mc Y_0$ is a {\em white noise} of variance $\sigma^2$ if for any $u \in \mc S(\bb R)$, the real-valued random variable $\mc Y_0(u)$ has a normal distribution of mean zero and variance $\sigma^2\|u\|^2$. Straightforward computations show that \eqref{ec1} has a unique invariant distribution, given by a white noise of variance $\sigma^2/2D$. We say that a solution $\{\mc Y_t; t \in [0,T]\}$ of \eqref{ec1} is {\em stationary} if $\mc Y_0$ is a white noise of variance $\sigma^2/2D$.

From now on, we fix a stationary solution $\{\mc Y_t; t \in [0,T]\}$ of \eqref{ec1}. We will denote by $\bb P$ the underlying probability measure and by $\bb E$ the expectation with respect to $\bb P$. First we notice that $\mc Y_t(u)$ is well defined for any $u \in L^2(\bb R)$: it is enough to take $\{u^n;  n \in \bb N\} \subseteq \mc S(\bb R)$ such that $\|u^n-u\| \to 0$ as $n \to \infty$ and to notice that $\{\mc Y_t(u^n);n \in \bb N\}$ is a Cauchy sequence in $L^2(\bb P)$. Let $h: \bb R \to \bb R$ be the function $h(x):=\mathbbm 1(0 <x \leq 1)$. For $\varepsilon \in (0,1)$ we denote by $i_\varepsilon$ the function $x  \mapsto \varepsilon^{-1} h(x\varepsilon^{-1})$. The sequence $\{i_\varepsilon; \varepsilon \in (0,1)\}$ is an {\em approximation of the identity} in the sense that for any $f \in \mc S(\bb R)$, $\<i_\varepsilon,f\> \to f(0)$ as $\varepsilon \to \infty$. For $x \in \bb R$, we will use the notation $i_\varepsilon(x)$ for the function $y \mapsto \varepsilon^{-1} h((y-x)\varepsilon^{-1})$.

 For each $\varepsilon \in (0,1)$, let us define the real-valued process $\{\mc Z_t^\varepsilon; t \in [0,T]\}$ as
\begin{equation}
\label{fBM}
 \mc Z_t^\varepsilon = \int_0^t \mc Y_s(i_\varepsilon) ds.
\end{equation}

\begin{theorem}
 \label{t1} The process $\{\mc Z_t^\varepsilon; t \in [0,T]\}$ converges in distribution with respect to the uniform topology of $\mc C([0,T], \bb R)$, as $\varepsilon \to 0$, to a fractional
Brownian motion $\{\mc Z_t; t \in [0,T]\}$ of Hurst exponent $H = 3/4$.
\end{theorem}

Let us denote by $\delta_0$ the Dirac distribution at $x=0$. This theorem is telling us that, although $\mc Y_t(\delta_0)$ is not well defined, the integral $\int_0^t \mc Y_s(\delta_0)ds$ can be defined through an approximation procedure.

For each $\varepsilon \in (0,1)$, let us define now the $\mc S'(\bb R)$-valued process $\{\mc A_t^\varepsilon; t \in [0,T]\}$ as
\begin{equation}
\label{defdea}
\mc A_t^\varepsilon(u) = \int_0^t \!\!\int\limits_{\bb R} \Big\{\mc Y_s(i_\varepsilon(x))^2-\frac{\sigma^2}{2D\varepsilon}\Big\}  u(x) dx ds
\end{equation}
for any $u \in \mc S(\bb R)$ and any $t \in [0,T]$.

\begin{theorem}
 \label{t2}
There exists a process $\{\mc A_t; t \in [0,T]\}$ with trajectories in $\mc C([0,T],\mc S'(\bb R))$ such that $\{\mc A_t^\varepsilon; t \in [0,T]\}$ converges in distribution with respect to the uniform topology of $\mc C([0,T],\mc S'(\bb R))$, as $\varepsilon \to 0$, to the process $\{\mc A_t; t \in [0,T]\}$.
\end{theorem}

This theorem was first proved in \cite{Ass}. In Section \ref{s4.2} we will present a different proof of this theorem, based on the so-called {\em energy condition} introduced in \cite{GJ}. We call the process $\{\mc A_t; t \in [0,T]\}$ the {\em quadratic field} associated to $\{\mc Y_t; t \in [0,T]\}$.

\subsection{Scaling limits}
\label{s1.4}
The lattice gas model introduced in Section \ref{s1.1} and the mean-zero exclusion process introduced in Section \ref{s1.2} are examples of diffusive systems. Various scaling limits of the density of particles have been established for those systems, and a common feature of any of these results is a diffusive scaling of the time variable. Let us focus on the lattice gas model of Section \ref{s1.1}. The so-called {\em hydrodynamic limit} of this model was established in \cite{FUY}. This hydrodynamic limit can be understood as a law of large numbers for the number of particles on boxes of increasing size. Once a law of large numbers has been established, a question which arises naturally is what happens with the fluctuations of the number of particles on boxes of increasing size. This question was answered in \cite{Cha, Sel} in the following sense. Let $n \in \bb N$ denote a scaling parameter. Let us introduce the {\em density fluctuation field} as the $\mc S'(\bb R)$-valued process $\mc Y_t^n$ given by
\[
 \mc Y_t^n(u) = \frac{1}{\sqrt n} \sum_{x \in \bb Z} \big(\eta_{tn^2}(x) -\rho\big) u(x/n)
\]
for any $u \in \mc S(\bb R)$.
Notice the diffusive time scaling in this definition.
Under the measure  $\nu_\rho$, for any $t>0$ the random variables $\{\eta_{tn^2}(x)-\rho; x \in \bb Z\}$ form an i.i.d.~sequence of mean zero and variance $\rho(1-\rho)$. Following \cite{BD-SGJ-LL}, we call $\chi(\rho)=\rho(1-\rho)$ the {\em mobility} of the system. In particular, for any $t>0$ fixed, the $\mc S'(\bb R)$-valued random variable $\mc Y_t^n$ converges in distribution, as $n \to \infty$ to a white noise of variance $\chi(\rho)$.
Let us introduce the diffusion coefficient
\[
 D(\rho) = \frac{1}{\chi(\rho)} \inf_{ f \text{ local}} \<r, \big(\eta(1)-\eta(0)-\nabla_{0,1} \sum_{x \in \bb Z} \tau_x f \big)^2\>.
\]

\begin{proposition}[\cite{Cha,Sel}]
\label{p1}
 The process $\{\mc Y_t^n;t \in [0,T]\}$ converges in distribution with respect to the $J_1$-Skorohod topology of $\mc D([0,T],\mc S'(\bb R))$ to the stationary solution of the infinite-dimensional Ornstein-Uhlenbeck equation
\begin{equation}
\label{ec2}
 d\mc Y_t = D(\rho) \Delta \mc Y_t dt + \sqrt{2D(\rho) \chi(\rho)} \nabla d \mc M_t.
\end{equation}
\end{proposition}


We are interested in functional limit theorems for certain observables of the process $\{\eta_t; t \geq0\}$ (and/or its rescaled version $\{\eta_{tn^2}; t \geq 0\}$). The occupation time of a site $x \in \bb Z$ is defined as the integral $\int_0^t \eta_s(x) ds$. The following theorem establishes the scaling limit of the occupation time of a fixed site $x \in \bb Z$.

\begin{theorem}
 \label{t3}
The process $\{\Gamma_t^n; t \in [0,T]\}$ defined as
\[
 \Gamma_t^n = \frac{1}{n^{3/2}} \int_0^{tn^2} \big(\eta_{s}(0)-\rho\big) ds
\]
converges in distribution with respect to the uniform topology of $\mc C([0,T],\bb R)$ to a fractional Brownian motion $\{\mc Z_t; t \in [0,T]\}$ of Hurst exponent $H=3/4$. Moreover, the process $\{\mc Z_t; t \in [0,T]\}$ is the same appearing in Theorem \ref{t1}.
\end{theorem}

Notice that by translation invariance the same result holds true for the occupation time of any site $x \in \bb Z$.
The occupation time is a particular case of a more general family of observables of the process $\{\eta_t; t \in [0,T]\}$.
Let $f: \Omega \to \bb R$ be a local function and define, for $\beta \in [0,1]$, $\varphi_f(\beta) = \int f d\nu_\beta$. The (centered) additive functional associated to $f$ is defined as the integral
\[
 \int_0^t \big(f(\eta_s)-\varphi_f(\rho)\big) ds.
\]

This process satisfies a functional central limit theorem similar to the one satisfied by the occupation time.

\begin{theorem}
 \label{t4}
The process $\{\Gamma_t^n(f); t \in [0,T]\}$ defined as
\[
 \Gamma_t^n(f) = \frac{1}{n^{3/2}} \int_0^{tn^2} \big(f(\eta_{s})-\varphi_f(\rho)\big) ds
\]
converges in distribution with respect to the uniform topology of $\mc C([0,T],\bb R)$ to the process $\{\varphi'_f(\rho) \mc Z_t; t \in [0,T]\}$, where the process $\{\mc Z_t; t \in [0,T]\}$ is the same appearing in Theorem \ref{t1}.
\end{theorem}

Notice that $\Gamma_t^{n}(f)$ looks only at the behavior of the process $\{\eta_t; t \geq 0\}$ on a neighborhood of the origin. An extensive field can be associated to the function $f$ in the following way. Let $\{ \Lambda_t^{n,f}; t \in [0,T]\}$ be the $\mc S'(\bb R)$-valued process defined as
\[
 \Lambda_t^{n,f}(u) = \int_0^t \sum_{x \in \bb Z} \big(\tau_x f(\eta_{sn^2})-\varphi_f(\rho)\big) u(x/n) ds
\]
for any $u \in \mc S(\bb R)$.
\begin{theorem}
 \label{t5}
Let us assume that $\varphi_f'(\rho)=0$. Then the  process $\{\Lambda_t^{n,f}; t \in [0,T]\}$ converges in distribution with respect to the uniform topology of $\mc C([0,T],\mc S'(\bb R))$ to the process $\{\frac{1}{2}\varphi_f''(\rho) \mc A_t; t \in [0,T]\}$, where $\{ \mc A_t; t \in [0,T]\}$ is the quadratic field associated to the stationary solution of \eqref{ec2}.
\end{theorem}

\section{Some auxiliary results}
\label{s2}

In this section we review some well known results in the theory of Markov chains, which will be the building blocks of the proofs of Theorems \ref{t3}, \ref{t4} and \ref{t5}. Since we did not find exact references, we give proofs of some of these results.

\subsection{Equivalence of ensembles}

Let $f:\Omega \to \bb R$ be a local function. For a finite set $A \subseteq \bb Z$, we say that $\supp(f) \subseteq A$ if $f(\eta(x))=f(\xi(x))$ whenever $\eta(x) =\xi(x)$ for any $x \in A$.
Assume for simplicity that there exists $\ell_0 \in \bb N$ such that $\supp(f) \subseteq \{1,...,\ell_0\}$. For $\ell \geq \ell_0$ we define the function $\psi_f(\ell):\Omega \to \bb R$ as
\[
 \psi_f(\ell;\eta) = E\Big[f\Big|\sum_{x=1}^\ell \eta(x)\Big],
\]
where the conditional expectation is taken with respect to $\nu_\rho$. When $f:\Omega \to \bb R$ is local, there always exist numbers $x \in \bb Z$ and $\ell_0 \in \bb N$ such that $\supp(\tau_x f) \subseteq \{1,...,\ell_0\}$. In that case we define, for $\ell \geq \ell_0$, $\psi_f(\ell) = \tau_{-x} \psi_{\tau_xf}(\ell)$. Notice that the definition of $\psi_f(\ell)$ depends on the choice of $x$ (it does not depend on the choice of $\ell_0$ as soon as $\ell \geq \ell_0$), but in that case the corresponding functions only differ by a translation. This point does not really matter, as soon as the values of $x$ and $\ell_0$ are kept fixed for the purposes of this definition. In order to make the notation simpler, we will always assume that $f$ is such that $\supp(f) \subseteq \{1,...,\ell_0\}$ for some $\ell_0 \in \bb N$.

Recall the definition $\varphi_f(\beta) = \int f d\nu_\beta$. For $\ell \in \bb N$ and $\eta \in \Omega$, let us define $\eta^\ell = \ell^{-1} \sum_{x=1}^\ell \eta(x)$. In other words, $\eta^\ell$ is the average density of particles on the interval $\{1,...,\ell\}$ with respect to the configuration $\eta$. The following proposition gives an approximation of $\psi_f(\ell)$ in terms of $\varphi_f$ and $\eta^\ell$.

\begin{proposition}[Equivalence of ensembles]
 \label{p5}
  Let $f:\Omega \to \bb R$ be a local function. Suppose that there exists $\ell_0 \in \bb N$ such that $\supp(f) \subseteq \{1,...,\ell_0\}$. There exists a constant $c=c(f)$ such that
\[
 \sup_{\eta \in \Omega} \Big| \psi_f(\ell; \eta) - \varphi_f(\eta^\ell) -\frac{\chi(\eta^\ell)}{2\ell} \varphi_f''(\eta^\ell)\Big| \leq \frac{c}{\ell^2}
\]
for any $\ell \geq \ell_0$.
\end{proposition}

\begin{proof}
Notice that, due to the structure of the configuration space $\Omega$, any local function $f$ is a finite linear combination of functions of the form $\prod_{x \in A} \eta(x)$, where $A \subseteq \bb Z$ is finite. It is enough to prove the theorem for functions of this form, since $\varphi_f$ and $\psi_f(\ell)$ are linear functions of $f$. The measure $\nu_\rho$ and the function $\eta^\ell$ are exchangeable, in the sense that they are left invariant by a permutation of the sites $\{1,...,\ell\}$. Therefore, it is enough to prove the theorem for the function $f= \prod_{x=1}^k \eta(x)$ for each $k \in \bb N$. Fix $k \in \bb N$ and take $\ell \geq k$. We can explicitly compute $\psi_f(\ell)$:
\[
 \psi_f(\ell;\eta) = \prod_{i=1}^{k-1}\frac{\ell}{\ell-i} \prod_{j=0}^{k-1} \Big(\eta^\ell -\frac{j}{\ell}\Big).
\]
Let us call $a_{\ell,k} = \prod_{i=0}^{k-1} \ell/(\ell-i)$. Notice that $a_{\ell,k}$ is uniformly bounded in $\ell$, and also notice that $a_{\ell,k} \to 1$ as $\ell \to \infty$. Developing the product $\prod_{j=0}^{k-1}(\eta^\ell-j/\ell)$ we obtain the expansion
\begin{equation}
\label{ec3}
 \psi_f(\ell;\eta) = a_{\ell,k} \sum_{j=0}^k \frac{p_j (\eta^\ell)^{\ell-j}}{\ell^j},
\end{equation}
where the coefficients $p_j$ do not depend on $\ell$. Notice that $0 \leq \eta^\ell \leq 1$ for any $\eta \in \Omega$. Therefore, all the powers of order smaller than $\ell-1$ in \eqref{ec3} are uniformly bounded by $c/\ell^2$, where the constant $c$\footnote{
As usual, $c$ will denote a constant which may change from line to line, but depends only on the specified parameters.}
depends only on the coefficients $p_j$ (and therefore it does not depend on $\ell$). We conclude that there exists a constant $c$ such that
\begin{equation}
\label{ec4}
 \sup_{\eta \in \Omega} \Big| \psi_f(\ell;\eta) -a_{\ell,k}\Big(p_0 (\eta^\ell)^k +\frac{p_1 (\eta^\ell)^{k-1}}{\ell}\Big)\Big| \leq \frac{c}{\ell^2}.
\end{equation}
The constants $p_0, p_1$ are easy to compute. Since each factor in $\prod_{j=0}^{k-1}(\eta^\ell-j/\ell)$ is monic, $p_0=1$ and $p_1/\ell$ is equal to minus the sum of the roots of each of such monomials. Therefore, $p_1 = -k(k-1)/2$. Now we need to expand $a_{\ell,k}$ in powers of $\ell$ up to order $2$. It is easier to expand $a_{\ell,k}^{-1}$. In fact,
\[
 a_{\ell,k}^{-1} = \sum_{i=0}^{k-1} \frac{p_i}{\ell^i} = 1 -\frac{k(k-1)}{2\ell} +\frac{r(\ell)}{\ell^2},
\]
where $r(\ell)$ is bounded in $\ell$. Using the expansion $(1-\delta+O(\delta^2))^{-1} = 1+ \delta +O(\delta^2)$ we conclude that
\[
 a_{\ell,k} = 1 +\frac{k(k-1)}{2\ell} +\frac{\tilde r(\ell)}{\ell^2}
\]
for another bounded function $\tilde r(\ell)$. Putting this expansion back into \eqref{ec4} we conclude that there exists a constant $c$ which depends only on $k$ such that
\[
  \sup_{\eta \in \Omega} \Big| \psi_f(\ell;\eta) - \Big(1+ \frac{k(k-1)}{2\ell}\Big)(\eta^\ell)^k +\frac{k(k-1)}{2\ell} (\eta^\ell)^{k-1}\Big| \leq \frac{c}{\ell^2}.
\]
For this particular choice of the function $f$ we have $\varphi_f(\eta^\ell) = (\eta^\ell)^k$ and $\varphi_f''(\eta^\ell) = k(k-1) (\eta^\ell)^{k-2}$. Replacing above $(\eta^\ell)^k$ by $\varphi_f(\eta^\ell)$ and $k(k-1) (\eta^\ell)^{k-1}$ by $\eta^\ell \varphi_f''(\eta^\ell)$, the proposition is proved.	
\end{proof}

Proposition \ref{p5} has the following consequence:

\begin{proposition}
\label{p6}
Let $f:\Omega \to \bb R$ be a local function. Suppose that there exists $\ell_0$ such that $\supp(f) \subseteq \{1,...,\ell_0\}$. Then there exists a constant $c=c(f,\rho)$ such that
\[
\int\Big( \psi_f(\ell,\eta)-\varphi_f(\rho)-\varphi'_f(\rho)\big(\eta^\ell-\rho\big) -\frac{\varphi''_f(\rho)}{2}\Big(\big(\eta^\ell-\rho\big)^2-\frac{\chi(\rho)}{\ell}\Big)\Big)^2 d\nu_\rho \leq \frac{c}{\ell^3}
\]
for any $\ell \geq \ell_0$. In particular,
\begin{itemize}
\item[i)] if $\varphi_f(\rho)=0$, we can choose $c$ in such a way that
\[
\Var(\psi_f(\ell);\nu_\rho) \leq \frac{c}{\ell},
\]
\item[ii)] if $\varphi_f(\rho) = \varphi'_f(\rho)=0$, we can choose $c$ in such a way that
\[
\Var(\psi_f(\ell);\nu_\rho) \leq \frac{c}{\ell^2},
\]
\end{itemize}
for any $\ell \geq \ell_0$.
\end{proposition}

Here and below, $\Var(f;\nu_\rho)$ denotes the variance of the function $f$ with respect to the probability measure $\nu_\rho$.

\begin{proof}
It is enough to use Taylor expansion in Proposition \ref{p5} up to the right order. Notice that
\[
\varphi_f(\eta^\ell) = \varphi_f(\rho)+\varphi_f'(\rho)\big(\eta^\ell-\rho\big) +\frac{\varphi_f''(\rho)}{2}\big(\eta^\ell-\rho\big)^2 + \delta_1(\eta^\ell;\rho)\big(\eta^\ell-\rho\big)^3,
\]
where $\delta_1(\eta^\ell;\rho)$ is uniformly bounded in $\eta$, $\ell$ and $\rho$. Similarly,
\[
\frac{\chi(\eta^\ell)}{2\ell} \varphi_f''(\eta^\ell) = \frac{1}{2\ell} \chi(\rho)\varphi_f''(\rho)+\delta_2(\eta^\ell;\rho) \frac{1}{\ell} \big(\eta^\ell-\rho\big),
\]
where $\delta_2(\eta^\ell;\rho)$ is uniformly  bounded in $\eta$, $\ell$ and $\rho$. The $L^2(\nu_\rho)$-norms of $(\eta^\ell-\rho)^3$ and $\frac{1}{\ell}(\eta^\ell-\rho)$ are bounded by $1/\ell^{3/2}$, which proves the first estimate. The other two cases are simple consequences of the first estimate.
\end{proof}

\subsection{Spectral tools}
\label{s2.1}
For $f, g \in L^2(\nu_\rho)$ we denote $\<f,g\>_\rho = \int fg d\nu_\rho$.
Let $f \in L^2(\bb R)$ be such that $\varphi_f(\beta) =0$ for any $\beta \in [0,1]$. Let us define the $H_{-1}$-norm of $f$ as
\begin{equation}
\label{h-1}
 \|f\|_{-1}^2 = \sup_{g \text{ local}}\{ 2\<f,g\>_\rho - \<g,-Lg\>_\rho\}.
\end{equation}
In \cite{KV}, the authors obtained a sharp estimate for the variance of the additive functional associated to $f$:

\begin{proposition}[Kipnis-Varadhan inequality \cite{KV,CLO}]
 \label{p2}
 For any $t \geq 0$,
\begin{equation}
\label{kv}
\bb E_\rho\Big[\Big( \int_0^t f(\eta_s)ds\Big)^2\Big] \leq 18 t \|f\|_{-1}^2.
\end{equation}
\end{proposition}

The original proof of \cite{KV} works only when $\nu_\rho$ is reversible with respect to the evolution of $\eta_t$. A proof avoiding reversibility can be found in \cite{CLO} (see also \cite{SVY}).

This estimate is not really helpful unless we have an efficient way to estimate $\|f\|_{-1}^2$. A very useful tool to estimate $\|f\|_{-1}^2$ is the so-called {\em spectral gap inequality} for the operator $L$. For $\ell \in \bb N$ and a local function $f:\Omega \to \bb R$ such that $\supp(f) \subseteq \{1,...,\ell\}$ we define the {\em energy form} $\mc E_\ell(f)$ as
\[
 \mc E_\ell(f) = \sum_{x=1}^{\ell-1} \int \big(\nabla_{x,x+1} f\big)^2 d\nu_\rho.
\]

Notice that in the case $r \equiv 1$, $\mc E_\ell(f) = \<f,-L_\ell f\>_\rho$, where $L_\ell$ is the generator of the process $\eta_t$ restricted to the interval $\{1,...,\ell\}$. Moreover, under the ellipticity condition we have $\varepsilon_0 \mc E_\ell(f) \leq \<f,-L_\ell f\>_\rho \leq \varepsilon_0^{-1} \mc E_\ell (f)$.

\begin{proposition}[Spectral gap inequality \cite{Qua,DS-C}]
\label{p3}
 There exists a universal constant $\kappa_0$ such that
\[
 \int f^2 d\nu_\rho \leq \kappa_0 \ell^2 \mc E_\ell(f)
\]
for any $\ell \in \bb N$ and any $f:\Omega \to \bb R$ such that $\varphi_f(\beta)=0$ for any $\beta \in [0,1]$ and such that $\supp(f) \subseteq \{1,...,\ell\}$.
\end{proposition}

By translation invariance, a similar statement holds when the support of the local function $f$ is contained in some interval of the form $\{k+1,...,k+\ell\}$.
The following proposition explains how to use the spectral gap inequality in order to estimate $\|f\|_{-1}^2$.

\begin{proposition}
 \label{p4}
For any function $f$ satisfying the hypothesis of Proposition \ref{p3},
\[
 \|f\|_{-1}^2 \leq \frac{\kappa_0}{\varepsilon_0} \ell^2 \Var(f;\nu_\rho).
\]
\end{proposition}
\begin{proof}
 Define $\mc F_\ell = \sigma(\eta(x); x=1,...,\ell)$. Let $g: \Omega \to \bb R$ be a local function and define $g_\ell = E[g|\mc F_\ell]$. Here and below, conditional expectations are taken with respect to the measure $\nu_\rho$. By Jensen's inequality, $\<g_\ell,-L_\ell g_\ell\>_\rho \leq \<g,-L_\ell g\>_\rho \leq \<g,-L g\>_\rho$. Since $f$ is $\mc F_\ell$-measurable, $\<f,g\>_\rho = \<f,g_\ell\>_\rho$. Therefore,
\[
 2\<f,g\>_\rho -\<g,-Lg\>_\rho \geq 2\<f,g_\ell\>_\rho -\<g_\ell,-L_\ell g_\ell\>_\rho.
\]
Let us define $\bar g = g_\ell -\psi_g(\ell)$. The function $\bar g$ satisfies $\varphi_{\bar g}(\beta)=0$, for any $\beta \in [0,1]$. Since we also have $\int f d\nu_\beta =0$ for any $\beta \in [0,1]$, we see that $\<f,g_\ell\>_\rho = \<f,\bar g\>_\rho$. Using once more Jensen's inequality, $\<\bar g,-L_\ell \bar g \>_\rho \leq \<g_\ell, -L_\ell g_\ell\>_\rho$. Therefore,
\[
 2\<f,g\>_\rho -\<g,-Lg\>_\rho \geq 2\<f,\bar g\>_\rho -\<\bar g,-L_\ell \bar g\>_\rho.
\]
Notice that $\bar g$ satisfies the hypothesis of Proposition \ref{p3}. Therefore,
\[
\<\bar g,-L_\ell \bar g\>_\rho \geq \varepsilon_0 \mc E_\ell(g) \geq \frac{\varepsilon_0}{\kappa_0 \ell^2} \Var(\bar g;\nu_\rho).
\]
We conclude that
\[
 \begin{split}
 \|f\|_{-1}^{2}:=\sup_{g \text{ local}} \big\{ 2\<f,g\>_\rho -\<g,-Lg\>_\rho\big\}
    &\leq \sup_{g \in \mc F_\ell} \big\{ 2\<f,\bar g\>_\rho -\<\bar g,-L_\ell \bar g\>_\rho\big\}\\
    &\leq \sup_{g \in \mc F_\ell} \big\{ 2\<f,\bar g\>_\rho -\frac{\varepsilon_0}{\kappa_0 \ell^2} \Var(\bar g;\nu_\rho)\big\}\\
    &\leq \frac{\kappa_0 \ell^2}{\varepsilon_0} \Var(f;\nu_\rho).\\
 \end{split}
\]
In last inequality we used the Cauchy-Schwarz inequality.
This ends the proof.
\end{proof}

The following proposition roughly states that functions supported on disjoint intervals are orthogonal with respect to the $H_{-1}$-norm.

\begin{proposition}
\label{p7}
Let $m \in \bb N$ be given. Take a sequence $k_0 < ...<k_m$ in $\bb Z$ and let $\{f_1,...,f_m\}$ be a sequence of local functions from $\Omega$ to $\bb R$ such that $\supp(f_i) \subseteq \{k_{i-1}+1,...,k_i\}$ for any $i \in \{1,...,m\}$. Define $\ell_i = k_i-k_{i-1}$. Assume that $\varphi_{f_i}(\beta)=0$ for any $\beta \in [0,1]$ and any $i \in \{1,...,m\}$. Then,
\[
\|f_1+...+f_m\|_{-1}^2 \leq \sum_{i=1}^m \frac{\kappa_0 \ell_i^2}{\varepsilon_0} \Var(f_i;\nu_\rho).
\]
\end{proposition}
\begin{proof}
Let us define $\mc G_i = \sigma(\eta(x); x=k_{i-1}+1,...,k_i)$. Let $g:\Omega \to \bb R$ be a local function and define $g^i = E[g|\mc G_i]$.  Let $f:=f_1+\cdots+f_m$. We have that
\[
\<f,g\>_\rho = \sum_{i=1}^m \<f_i,g^i\>_\rho
\]
and by Jensen's inequality we have that
\[
\<g,-Lg\> \geq \sum_{i=1}^m \sum_{x=k_{i-1}+1}^{k_i-1} \int r_x\big( \nabla_{x,x+1}g)^2 d\nu_\rho \geq \sum_{i=1}^m \sum_{x=k_{i-1}+1}^{k_i-1} \int r_x\big( \nabla_{x,x+1}g^i)^2 d\nu_\rho.
\]
Let us write $L^i$ for the generator of the process $\eta_t$ restricted to the interval $\{k_{i-1}+1,...,k_i\}$.
Following the proof of Proposition \ref{p4} we obtain that
\[
\begin{split}
\|f\|_{-1}^2= \sup_{g \text{ local}} \big\{ 2\<f,g\>_\rho -\<g,-Lg\>_\rho\big\}
		&\leq \sup_{g \text{ local}} \sum_{i=1}^m \big\{ 2\<f_i,g^i\>_\rho -\<g^i,-L^ig^i\>_\rho\}\\
		&\leq \sum_{i=1}^m \sup_{g \in \mc G_i}\big\{  2\<f_i,\bar g\>_\rho - \frac{\varepsilon_0}{\kappa_0 \ell_i^2} \Var(\bar g;\nu_\rho)\big\}\\
		& \leq \sum_{i=1}^m \frac{\kappa_0 \ell^2_i}{\varepsilon_0} \Var(f_i;\nu_\rho).
\end{split}
\]
\end{proof}

Putting all the estimates of this section together we obtain the following estimate on the variance of sums of additive functionals of the process $\eta_t$:

\begin{proposition}
\label{p8}
Let $m \in \bb N$ be given. Take a sequence $k_0 < ...<k_m$ in $\bb Z$ and let $\{f_1,...,f_m\}$ be a sequence of local functions from $\Omega$ to $\bb R$ such that $\supp(f_i) \subseteq \{k_{i-1}+1,...,k_i\}$ for any $i \in \{1,...,m\}$. Define $\ell_i = k_{i}-k_{i-1}$. Assume that $\varphi_{f_i}(\beta)=0$ for any $\beta \in [0,1]$ and any $i \in \{1,...,m\}$. Then, for any $t \geq 0$
\[
\bb E_\rho\Big[\Big(\int_0^t \sum_{i=1}^m f_i(\eta_s) ds \Big)^2\Big] \leq \frac{18\kappa_0  t}{\varepsilon_0} \sum_{i=1}^m \ell^2_i\Var(f_i;\nu_\rho).
\]
\end{proposition}

\section{Boltzmann-Gibbs principle}
\label{s3}
The Boltzmann-Gibbs principle was introduced in \cite{BR} and roughly speaking, it says that in the diffusive scaling introduced in Section \ref{s1.4}, any extensive field can be approximated by a function of the density fluctuation field. More precisely, we have the following proposition:

\begin{proposition}[\cite{D-MPSW, Cha2}]
  \label{p3.1}
Let $f: \Omega \to \bb R$ be a local function. For any $u \in \mc S(\bb R)$ and any $t \geq 0$
\[
\lim_{n \to \infty} \bb E_\rho\Big[\Big(\int_0^t \frac{1}{\sqrt n} \sum_{x \in \bb Z} \Big(\tau_x f(\eta_{sn^2}) -\varphi_f(\rho) -\varphi'_f(\rho)\big(\eta_{sn^2}(x)-\rho\big)\Big)u(x/n) ds\Big)^2\Big] =0.
\]
\end{proposition}

Notice that the last term in the sum above is the field $\varphi_f'(\rho) \int_0^t \mc Y_s^n(u) ds$. This result was proved for the model introduced in Section \ref{s1.1} in \cite{D-MPSW}. Their proof is an adaptation of the proof in \cite{BR} and it requires the reversibility and the translation invariance of the measure $\nu_\rho$. In \cite{Cha2}, a simpler proof of Proposition \ref{p3.1} was obtained, building up in the {\em one-block estimate} introduced in \cite{GPV}. Following the methods in \cite{CLO}, the proof in \cite{Cha} can be adapted for the mean-zero exclusion process defined in Section \ref{s1.2}.

In this section we will state and prove two different generalizations of Proposition \ref{p3.1}. When $\varphi_f(\rho)= \varphi_f'(\rho)=0$, Proposition \ref{p3.1} does not give any useful information about the limiting behavior of the fluctuation field associated to the function $f$. In that case, we have the following result:

\begin{theorem}[Second-order Boltzmann-Gibbs principle]
\label{t6}
Let $f:\Omega \to \bb R$ be a local function. Assume that $\varphi_f(\rho)= \varphi_f'(\rho)=0$. There exists a constant $c=c(f,\rho)$ such that for any $t \geq 0$ , any $\ell \in \bb N$ and any function $h=\{h_x; x \in \bb Z\} \in \ell^2(\bb Z)$
\begin{equation}
\label{sobg}
\bb E_\rho\Big[\Big( \int_0^t \sum_{x \in \bb Z} \Big( \tau_x f(\eta_s) -\frac{\varphi_f''(\rho)}{2}\Big(\big(\tau_x\eta^\ell_s-\rho\big)^2-\frac{\chi(\rho)}{\ell}\Big)\Big)h_xds \Big)^2\Big] \leq c\Big(t\ell+\frac{t^2}{\ell^2}\Big)\sum_{x \in \bb Z} h_x^2.
\end{equation}
\end{theorem}

Notice that in this theorem a sort of ``quadratic'' field appears. In Section \ref{s4.1} we will see how this estimate implies Theorem \ref{t5}. The second generalization of Proposition \ref{p3.1} deals with the case in which there is no spatial average in the additive functional associated to $f$:

\begin{theorem}[Local Boltzmann-Gibbs principle]
\label{t7}
Let $f: \Omega \to \bb R$ be a local function. Assume that $\varphi_f(\rho)=0$. There exists a finite constant $c=c(f,\rho)$ such that

\begin{itemize}
\item[i)] if $\varphi_f'(\rho) \neq 0$, then
\begin{equation}
 \label{lbg1}
\bb E_\rho \Big[ \Big( \int_0^t \Big\{ f(\eta_s) - \varphi_f'(\rho) \big(\eta_s^\ell -\rho\big)\Big\}ds\Big)^2\Big] \leq c\Big(t\ell
+\frac{t^2}{\ell^2}\Big)
\end{equation}
for any $t \geq 0$ and any $\ell \in \bb N$,
\item[ii)] if $\varphi_f'(\rho)=0$, then
\begin{equation}
\label{lbg2}
\bb E_\rho\Big[\Big(\int_0^t\Big\{ f(\eta_s)-\frac{\varphi_f''(\rho)}{2} \Big(\big(\eta_s^\ell -\rho\big)^2-\frac{\chi(\rho)}{\ell}\Big)\Big\} ds\Big)^2\Big] \leq c \Big(t(\log \ell)^2 +\frac{t^2}{\ell^3}\Big).
\end{equation}
for any $t \geq 0$ and any $\ell \in \bb N$.
\end{itemize}
\end{theorem}

This estimate is sharp in the case i) and it is close to the expected upper bound $c t \log \ell $ in case ii). The main point that makes both Theorems \ref{t6} and \ref{t7} different from Proposition \ref{p3.1} is the following. The functional $\varphi_f(\rho) \int_0^t \mc Y_s^n(u) ds$ is {\em continuous and linear} as a function of the process$\{\mc Y_t; t \in [0,T]\}$. Therefore, aside from technical details, in order to prove Proposition \ref{p3.1} it is enough to prove the one-block estimate. This is not longer enough for Theorems \ref{t6} and \ref{t7}, where we also need to prove the so-called {\em two-blocks estimate}. In a similar context, a proof of the two-blocks estimate and also the derivation of Theorem \ref{t6} was recently obtained in \cite{GJ}. In Section \ref{s3.1} we adapt the arguments in \cite{GJ} in order to prove Theorem \ref{t7}. We will not prove Theorem \ref{t6} here, since the proof in \cite{GJ} applies to our situation with basically notational modifications.

\subsection{Local Boltzmann-Gibbs principle}
\label{s3.1}

In this section we prove Theorem \ref{t7}. For notational simplicity we assume that there exists $\ell_0$ such that $\supp(f) \subseteq \{1,...,\ell_0\}$. We point out that the proof will only use Propositions \ref{p6} and \ref{p8} as inputs. Otherwise the proof is independent of the results in Section \ref{s2}. The proof follows like in Section 4 of \cite{GJ}: we have to divide it into four steps. The first part is what we call the {\em one-block estimate} and it compares the additive functional associated to $f$ with the additive functional associated to $\psi_f(\ell)$ for any $\ell \geq \ell_0$. The second part is the renormalization step, which compares the additive functional associated to $\psi_f(\ell)$ with the additive functional associated to $\psi_f(2\ell)$ for any $\ell \geq \ell_0$. The third step is what we call the {\em two-blocks estimate}, which compares the additive functional associated to $\psi_f(\ell_0)$ with the additive functional associated to $\psi_f(2^m \ell_0)$, using the renormalization step $m$ times. And the fourth step simply replaces $\psi_f(\ell)$ by the corresponding function of $\eta^\ell$ using the equivalence of ensembles stated in Proposition \ref{p6}.

\begin{lemma}[One-block estimate]
\label{one-block}
Let $f$ be a local function such that $\varphi_f(\rho)=0$. Then, there exists a constant $c=c(f,\rho)$ such that for any $\ell \geq \ell_0$ and any $t \geq 0$
\begin{equation}
\label{ec5}
\bb E_\rho\Big[\Big(\int_0^t\{ f(\eta_s)-\psi_f(\ell;\eta_s)\} ds \Big)^2\Big] \leq c t \ell^2 \Var(f;\nu_\rho).
\end{equation}
\end{lemma}
\begin{proof}
By Proposition \ref{p8}, the left-hand side of \eqref{ec5} is bounded by
\[
\frac{18 \kappa_0 t \ell^2}{\varepsilon_0}\Var(f-\psi_f(\ell);\nu_\rho).
\]
Since $\psi_f(\ell)$ is a conditional expectation, $\Var(f-\psi_f(\ell);\nu_\rho)$ is bounded above by $\Var(f;\nu_\rho)$, which proves the lemma.
\end{proof}

\begin{lemma}[Renormalization step]
\label{renormalization}
Let $f$ be a local function such that $\varphi_f(\rho)=0$. There exists a constant $c=c(f,\rho)$ such that for any $\ell \geq \ell_0$ and any $t \geq 0$
\begin{equation}
\label{ec6}
\bb E_\rho\Big[\Big(\int_0^t\{ \psi_f(\ell;\eta_s)-\psi_f(2\ell;\eta_s)\} ds \Big)^2\Big] \leq
\begin{cases}
c t \ell& \text{ if } \varphi_f'(\rho) \neq 0,\\
c t & \text{ if } \varphi_f'(\rho) =0.
\end{cases}
\end{equation}
\end{lemma}
\begin{proof}
By Proposition \ref{p8}, the left-hand side of \eqref{ec6} is bounded by
\[
\frac{72 \kappa_0 t \ell^2}{\varepsilon_0}\Var(\psi_f(\ell)-\psi_f(2\ell);\nu_\rho).
\]
By Proposition \ref{p6}, this last variance is bounded by $3c/\ell$ if $\varphi_f(\rho)=0$ and by $3c/\ell^2$ if in addition $\varphi_f'(\rho)=0$, which proves the lemma.
\end{proof}

\begin{lemma}[Two-blocks estimate]
\label{two-blocks}
Let $f$ be a local function  such that $\varphi_f(\rho)=0$. Then, there exists a constant $c = c(f,\rho)$ such that for any $\ell \geq  \ell_0$ and any $t \geq 0$,
\[
\bb E_\rho\Big[\Big(\int_0^t\{ \psi_f(\ell_0;\eta_s)-\psi_f(\ell;\eta_s)\} ds \Big)^2\Big] \leq
\begin{cases}
c t \ell & \text{ if } \varphi_f'(\rho) \neq 0,\\
c t (\log \ell)^2 & \text{ if } \varphi_f'(\rho) =0.
\end{cases}
\]
\end{lemma}
\begin{proof}
Let us assume that there exists $m \in \bb N$ such that $ \ell = 2^m \ell_0$. We write
\[
\psi_f(\ell_0;\eta_s) - \psi_f(\ell;\eta_s) = \sum_{i=1}^m \big(\psi_f(2^{i-1} \ell_0;\eta_s) - \psi_f(2^{i}\ell_0;\eta_s)\big).
\]
Using the renormalization step (Lemma \ref{renormalization}) and Minkowski's inequality, we conclude that
\[
\begin{split}
\bb E_\rho\Big[\Big(\int_0^t\{ \psi_f(\ell_0;\eta_s)
  &-\psi_f(\ell;\eta_s)\} ds \Big)^2\Big]^{1/2} \leq \\
  & \leq \sum_{i=1}^m \bb E_\rho\Big[\Big(\int_0^t\{ \psi_f(2^{i-1} \ell_0;\eta_s)-\psi_f(2^{i}\ell_0;\eta_s)\} ds \Big)^2\Big]^{1/2}\\
  &\leq \sum_{i=1}^m \big\{c t (2^{i-1} \ell_0)^\alpha\big\}^{1/2},
\end{split}
\]
where $\alpha =1$ if $\varphi_f'(\rho) \neq 0$ and $\alpha=0$ if $\varphi_f'(\rho)=0$. When $\alpha=1$ the sum above is equal to $\sqrt{ct \ell_0} (\sqrt{2}^m-1)/(\sqrt 2-1)$. When $\alpha =0$, the sum above is equal to $\sqrt{ct} m$. These two facts  prove the theorem for $\ell = 2^m \ell_0$. For $\ell$ arbitrary, it is enough to find $m$ such that $2^{m-1}\ell_0 <\ell \leq 2^m \ell_0$ and to compare $\psi_f(2^{m-1}\ell_0)$ with $\psi_f(\ell)$ as in Lemma \ref{renormalization}.
\end{proof}

The following lemma is just a reformulation of Proposition \ref{p6} together with the Cauchy-Schwarz inequality.
\begin{lemma}
\label{lema4}
Let $f$ be a local function such that $\varphi_f(\rho)=0$. Then there exists a constant $c=c(f,\rho)$ such that for any $\ell \geq \ell_0$ and any $t \geq 0$
\begin{itemize}
\item[i)] if $\varphi_f'(\rho) \neq 0$, then
\[
\bb E_\rho\Big[\Big(\int_0^t\{ \psi_f(\ell;\eta_s) -\varphi_f'(\rho)\big(\eta_s^\ell-\rho\big)\} ds \Big)^2\Big] \leq \frac{ct^2}{\ell^2},
\]
\item[ii)] if $\varphi_f'(\rho)=0$, then
\[
\bb E_\rho\Big[\Big(\int_0^t\Big\{ \psi_f(\ell;\eta_s) - \frac{\varphi_f''(\rho)}{2}\Big(\big(\eta_s^\ell-\rho\big)^2-\frac{\chi(\rho)}{\ell}\Big)\Big\} ds \Big)^2\Big] \leq \frac{ct^2}{\ell^3}.
\]
\end{itemize}
\end{lemma}

Theorem \ref{t7} follows at once by combining these four lemmas. Notice that the use of Lemmas \ref{one-block}, \ref{renormalization}, \ref{two-blocks} gives rise to the first term at the right-hand side of inequalities \eqref{lbg1},\eqref{lbg2} and the use of Lemma \ref{lema4} gives rise to the second term. We state this observations as a corollary:

\begin{corollary}
 \label{cor1}
Under the hypothesis of Theorem \ref{t7},
\begin{equation}
 \label{est1}
\bb E_\rho\Big[\Big(\int_0^t\{ f(\eta_s)-\psi_f(\ell;\eta_s)\} ds \Big)^2\Big] \leq
\begin{cases}
c t \ell& \text{ if } \varphi_f'(\rho) \neq 0,\\
c t (\log \ell)^2& \text{ if } \varphi_f'(\rho) =0
\end{cases}
\end{equation}
for any $t \geq 0$ and any $\ell \geq \ell_0$.
\end{corollary}

\begin{remark}
It has been proved in \cite{SX} that for functions $f$ satisfying $\varphi_f(\rho) = \varphi_f'(\rho)=\varphi_f''(\rho)=0$, the variance of $\int_0^t f(\eta_s)ds$ is bounded by $c(f,\rho)t$. Moreover, they also show that the scaling limit of $\int_0^t f(\eta_s)ds$ is given by a Brownian motion of variance given in terms of  a variational formula.
\end{remark}

\section{Proofs}
\label{s4}
\subsection[Additive functionals and proof of Theorem 2.1]{Additive functionals and proof of Theorem \texorpdfstring{ \ref{t1}}{2.1}}
\label{s4.1}
In this section we prove Theorem \ref{t1}. For each $\varepsilon >0$ and each $n \in \bb N$, let us define
\[
\mc Z_t^{n,\varepsilon} = \frac{1}{n^{3/2}} \int_0^{t n^2} (\eta_s^{\varepsilon n} -\rho) ds = \int_0^{t} \mc Y_s^n(i_\varepsilon (0)) ds.
\]

Recall the definition of $\Gamma_t^n(f)$ given in Section \ref{s1.4}. Suppose for a moment that $\varepsilon n$ is an integer. Using Theorem \ref{t7} with $\ell = \varepsilon n$ for time $t n^2$ we obtain the following {\em energy estimate}:
\begin{equation}
\label{localenergy}
\bb E_\rho\big[ \big( \Gamma_t^n(f) - \varphi_f'(\rho) \mc Z_t^{n,\varepsilon}\big)^2\big] \leq c\Big\{ \varepsilon t + \frac{t^2}{\varepsilon^2 n}\Big\}.
\end{equation}
When $\varepsilon n$ is not an integer, there are some round-off errors, which can be absorbed into the choice of the constant $c$ whenever $\varepsilon n \geq 2$, for example.
By the triangle's inequality, we also have the following estimate:
\begin{equation}
\label{localenergy2}
\bb E_\rho\big[ \big( \mc Z_t^{n,\varepsilon}  -  \mc Z_t^{n,\delta}\big)^2\big] \leq c\Big\{ \varepsilon t + \frac{t^2}{\delta^2 n}\Big\}
\end{equation}
for any $t \geq$, any $n \in \bb N$ and any $\varepsilon > \delta>2/n$, for a (maybe different) constant $c$ which does not depend on $n,t,\varepsilon$ or $\delta$. Now we prove Theorem \ref{t1} starting from \eqref{localenergy2}. We start proving the following moment estimate.

\begin{lemma}
\label{l1} There exists a constant $c$ such that for any $\varepsilon >0$ and any $0<s<t$,
\[
\bb E\big[\big(\mc Z_t^\varepsilon-\mc Z_s^\varepsilon\big)^2\big] \leq c|t-s|^{3/2},
\]
where $\{\mc Z_t^\varepsilon; t \in [0,T]\}$ is the process defined in \eqref{fBM}.
\end{lemma}

\begin{proof}
Recall the convergence result stated in Proposition \ref{p1}. Since $L^2$-upper bounds are preserved by weak convergence, from \eqref{localenergy2} we deduce that there exists a constant $c$ such that
\begin{equation}
\label{en.est}
\bb E\big[\big(\mc Z_t^\varepsilon-\mc Z_t^{\delta}\big)^2\big] \leq c \varepsilon t
\end{equation}
for any $t \geq 0$ and any $\varepsilon>\delta>0$. Recall that for any fixed time $t \geq 0$, $\mc Y_t$ is a white noise of variance $\chi(\rho)$. By the Cauchy-Schwarz inequality, we have the elementary bound
\begin{equation}
\label{elem.est}
\bb E\big[\big(\mc Z_t^\varepsilon\big)^2\big] \leq \frac{\chi(\rho)t^2}{\varepsilon}.
\end{equation}
Fix $\delta >0$. For any $\varepsilon \geq \delta$ we have
\[
\bb E\big[\big(\mc Z_t^\delta\big)^2\big] \leq 2c\varepsilon t + \frac{2\chi(\rho) t^2}{\varepsilon}.
\]
For $t \geq \delta^2$, taking $\varepsilon = \sqrt t$ in the estimate above we prove that there is a constant $c$ independent of $t$ and $\delta$ such that
\begin{equation}
\label{moment}
 \bb E\big[\big(\mc Z_t^\delta\big)^2\big] \leq ct^{3/2}.
\end{equation}
For $t <\delta^2$, $t^2 \leq \delta t^{3/2}$. Therefore, taking $\varepsilon = \delta$ in \eqref{elem.est} we can extend the validity \eqref{moment} to any $t \geq 0$ and any $\delta > 0$. Since the process $\mc Y_t$ is stationary, the proof is finished.
\end{proof}

Now let us finish the proof of Theorem \ref{t1}. By Kolmogorov-Centsov's compactness criterion (see Problem 2.4.11 of \cite{KS}), the sequence of processes $\{\mc Z_t^\varepsilon; t \in [0,T]\}_{\varepsilon>0}$ is tight with respect to the uniform topology of $\mc C([0,T],\bb R)$. Therefore, it has at least one limit point $\{\mc Z_t; t \in [0,T]\}$. Moreover, by \eqref{en.est}, for any fixed time $t \geq 0$, the sequence $\{\mc Z_t^\varepsilon; \varepsilon >0\}$ is a Cauchy sequence in $L^2(\bb P)$. Therefore, for any $t \geq 0$, $\mc Z_t^\varepsilon$ converges in $L^2(\bb P)$ to $\mc Z_t$. This proves the uniqueness of the limit point $\{\mc Z_t; t \in [0,T]\}$ and the weak convergence with respect to the uniform topology of the whole sequence $\{\mc Z_t^\varepsilon; t \in [0,T]\}_{\varepsilon>0}$ to $\{\mc Z_t; t \in [0,T]\}$.
By the definition of $\mathcal{Z}_t^\varepsilon$ given in \eqref{fBM}, it is the integral of a Gaussian random variable $\mathcal{Y}_s(i_\varepsilon)$. Therefore $\mathcal{Z}_t^\varepsilon$ is a Gaussian random variable.
Since the space of Gaussian random variables is closed in $L^2$, then  $\mc Z_t$ is also Gaussian. The process $\mc Y_t$ has the following scale invariance: for any $F \in \mc S(\bb R)$ and any $\lambda >0$, $\mc Y_t(F(\cdot)) = \lambda^{1/2} \mc Y_{t\lambda^{-2}}(F(\lambda\cdot))$ in distribution. From this property, it follows that $\mc Z_t$ satisfies the following self-similarity relation: $\mc Z_{\lambda t} = \lambda^{3/4} \mc Z_t$ in distribution. These three properties, namely continuity, Gaussianity and $3/4$-self-similarity characterizes a fractional Brownian motion of Hurst exponent $H=3/4$. Therefore, Theorem \ref{t1} is proved for the case $D = D(\rho)$ and $\sigma^2 =2D(\rho)\chi(\rho)$. To show that Theorem \ref{t1} holds for any values of $D$ and $\sigma^2$, it is enough to observe that rescaling time and space properly we can obtain any value of $D$, $\sigma^2$.

\subsection[Proof of Theorems 2.4 and 2.5]{Proof of Theorems \texorpdfstring{ \ref{t3}}{2.4} and \texorpdfstring{ \ref{t4}}{2.5}}
In this section we prove Theorems \ref{t3} and \ref{t4}. Notice that Theorem \ref{t3} is just a particular case of Theorem \ref{t4}, so we will only prove Theorem \ref{t4}. We start proving tightness of the process $\{\Gamma_t^n; t \in [0,T]\}$. The proof will closely follow the proof of Lemma \ref{l1}. Using the stationarity of the process $\eta_t$ plus the Cauchy-Schwarz inequality, we see that there is a constant $c=c(\rho,f)$
\footnote{Up to the end of this section, we denote by $c$ any constant which depends only on $f$ and $\rho$. This constant may change from line to line.}
such that
\begin{equation}
\label{badbound}
\bb E_\rho\big[\Gamma_t^n(f)^2\big] \leq c t^2 n.
\end{equation}
This bound does not look too good, but it will be useful for very short times $t$. By Corollary \ref{cor1}, we also have the bound
\[
 \bb E_\rho\Big[\Big(\Gamma_t^n(f) - \frac{1}{n^{3/2}} \int_0^{tn^2} \psi_f(\ell;\eta_s)ds\Big)^2\Big] \leq \frac{ct \ell}{n}.
\]
By the Cauchy-Schwarz inequality, the stationarity of the process $\eta_t$ and Proposition \ref{p6}, we  have the bound
\[
 \bb E_\rho\Big[\Big(\frac{1}{n^{3/2}} \int_0^{tn^2} \psi_f(\ell;\eta_s)ds\Big)^2\Big] \leq \frac{c t^2 n}{\ell}
\]
for any $t \geq 0$ and any $\ell \geq \ell_0$. We conclude that the process $\Gamma_t^n(f)$ satisfies the moment bound
\[
 \bb E_\rho\big[\big(\Gamma_t^n(f)\big)^2\big] \leq c\Big\{ \frac{t \ell}{n} + \frac{t^2 n}{\ell}\Big\}
\]
for any $t \geq 0$ and any $\ell \geq \ell_0$. If we take $\ell = n \sqrt t$, we obtain the bound
\begin{equation}
\label{momentbound}
  \bb E_\rho\big[\big(\Gamma_t^n(f)\big)^2\big] \leq c t^{3/2},
\end{equation}
valid for any $t \geq \ell_0^2/n^2$. This last restriction comes from the restriction $\ell \geq \ell_0$. For $t \leq \ell^2/n^2$ we observe that $t^2 n \geq t^{3/2} \ell_0$. This observation allows to use \eqref{badbound} in order to extend \eqref{momentbound} to any $t \geq 0$. Since the process $\{\Gamma_t^n(f);t \in [0,T]\}$ has stationary increments, using Kolmogorov-Centsov's compactness criterion we conclude that the sequence of processes $\{\Gamma_t^n(f);t \in [0,T]\}_{n \in \bb N}$ is tight with respect to the uniform topology in $\mc D([0,T];\bb R)$. Let $\{\Gamma_t(f);t \in [0,T]\}$ be a limit point of $\{\Gamma_t^n(f);t \in [0,T]\}_{n \in \bb N}$. Without loss of generality, we can assume that the limit process $\Gamma_t(f)$ is defined in the same probability space on which the process $\{\mc Y_t;t \in [0,T]\}$ is defined. Recall that $L^2$-upper bounds are preserved by convergence in distribution. Therefore, from \eqref{localenergy} we conclude that
\[
 \bb E_\rho\big[ \big(\Gamma_t(f)- \varphi_f'(\rho)\mc Z_t^\varepsilon\big)^2\big] \leq ct\varepsilon.
\]
In particular, we conclude that $\{\Gamma_t(f);t \in [0,T]\}$ has the same finite-dimensional distributions of $\{\varphi_f'(\rho) \mc Z_t; t \in [0,T]\}$, which finishes the proof of Theorem \ref{t4}.

\subsection[Quadratic fields and the proof of Theorem 2.2]{Quadratic fields and the proof of Theorem \texorpdfstring{ \ref{t2}}{2.2}}
\label{s4.2}
In this section we prove Theorem \ref{t2}, starting from the
second-order Boltzmann-Gibbs principle stated in Theorem \ref{t6}. We will need the following {\em energy estimate} (proved in Corollary 3.11 of \cite{GJ}):

\begin{proposition}
 \label{p9}
Under the hypothesis of Theorem \ref{t6},
\begin{equation}
 \label{energy}
\bb E_\rho\Big[ \Big(\int_0^t \sum_{x \in \bb Z}\tau_x \big(f(\eta_s) - \psi_f(\ell;\eta_s)\big)h_x ds\Big)^2\Big] \leq ct\ell \sum_{x \in \bb Z} h_x^2.
\end{equation}
\end{proposition}

The spirit of the proof of Theorem \ref{t2} is basically the same of the proof of Theorem \ref{t1}. For $x \in \bb Z$, $\ell \in \bb N$ and $\eta \in \Omega$, let us define $\eta^\ell(x) = \tau_x \eta^\ell = \ell^{-1} \sum_{i=1}^\ell \eta(x+i)$. In other words, $\eta^\ell(x)$ is the average density of particles on the interval $\{x+1,...,x+\ell\}$. For each $\varepsilon >0$ and each $n \in \bb N$, let us define the field $\{\mc A_t^{n,\varepsilon};t \in [0,T]\}$ as the $\mc S'(\bb R)$-valued process given by
\[
 \mc A_t^{n,\varepsilon}(u) = \frac{1}{n^2}\int_0^{tn^2} \sum_{x \in \bb Z} \Big(\big(\eta_s^{\varepsilon n}(x) -\rho\big)^2 -\frac{\chi(\rho)}{\varepsilon n}\Big) u(x/n) ds
\]
for any test function $u \in \mc S(\bb R)$ and any $t \in [0,T]$. Recall the definition of the process $\{\mc A_t^{\varepsilon};t \in [0,T]\}$ given in \eqref{defdea}. Since $\varepsilon >0$ is fixed, the convergence of the density field $\{\mc Y_t^n; t \in [0,1]\}$ ensures that
\[
 \lim_{n \to \infty} \mc A_t^{n,\varepsilon} = \mc A_t^\varepsilon
\]
in the sense of finite-dimensional distributions. Using triangle's inequality and \eqref{sobg} we obtain the following bound:
\[
 \bb E_\rho\big[ \big(\mc A_t^{n,\varepsilon}(u) - \mc A_t^{n,\delta}(u)\big)^2\big] \leq c\Big\{ t \varepsilon + \frac{t^2}{\delta^2 n}\Big\} \frac{1}{n} \sum_{x \in \bb Z} u(x/n)^2
\]
for any $0 < \delta < \varepsilon$, any $t \in [0,T]$ and any $u \in \mc S(\bb R)$. Taking the limit as $n \to \infty$, we obtain that
\begin{equation}
\label{enest}
 \bb E \big[\big(\mc A_t^\varepsilon(u)- \mc A_t^\delta(u)\big)^2\big] \leq c\varepsilon t \|u\|^2
\end{equation}
for any $0<\delta<\varepsilon$ and any $t \in [0,T]$. This inequality is what we call the {\em energy estimate} for the process $\{\mc Y_t; t \in [0,T]\}$. Notice that $\{\mc A_t^\varepsilon;\varepsilon>0\}$ is a Cauchy sequence in $L^2(\bb P)$. In particular,
for each fixed $t \in [0,T]$ and $u \in \mc S(\bb R)$ the random variable $\mc A_t(u) = \lim_{\varepsilon \to 0} \mc A_t^\varepsilon(u)$ is well defined. Notice that this does not imply the existence of the {\em process} $\{\mc A_t;t \in [0,T]\}$, neither the convergence of $\mc A_t^\varepsilon$ to $\mc A_t$ at the process level. However, if we prove that the sequence of processes $\{\mc A_t^\varepsilon; t \in [0,T]\}_{\varepsilon >0}$ is tight, then the existence and uniqueness of the process $\{\mc A_t;t \in [0,T]\}$ will be guaranteed. In order to prove tightness of a sequence of $\mc S'(\bb R)$-valued processes, the following criterion, known as Mitoma's criterion is very useful.

\begin{proposition}[Mitoma's criterion \cite{Mit}]
 \label{Mitoma} The sequence $\{\mc A_t^\varepsilon; t \in [0,T]\}_{\varepsilon >0}$ is tight with respect to the uniform topology of $\mc C([0,T];\mc S'(\bb R))$ if and only if for any $u \in \mc S(\bb R)$, the sequence $\{\mc A_t^\varepsilon(u); t \in [0,T]\}_{\varepsilon >0}$ is tight with respect to the uniform topology of $\mc C([0,T];\bb R)$.
\end{proposition}

Using the Cauchy-Schwarz inequality and the stationarity of the process $\mc Y_t$, we obtain the simple bound
\[
 \bb E\big[\mc A_t^\delta(u)^2\big] \leq ct^2\delta^{-1} \|u\|^2
\]
for any $t \in [0,T]$ and any $\delta>0$. Taking $\delta =\sqrt t$ in this estimate and in the energy estimate \eqref{enest} we obtain the bound
\[
 \bb E\big[\mc A_t^\varepsilon(u)^2\big] \leq c t^{3/2}
\]
for any $t \geq \varepsilon^2$. For $t \leq \varepsilon^2$, we observe that $t^2 \leq \varepsilon t^{3/2}$, which extends the bound for any $t \in [0,T]$. Since the process $\mc A_t^\varepsilon(u)$ has stationary increments, using Kolmogorov-Centsov's compactness criterion (see Problem 2.4.11 of \cite{KS}) we conclude that the sequence $\{\mc A_t^\varepsilon(u); t \in [0,T]\}_{\varepsilon >0}$ is tight, which finishes the proof of Theorem \ref{t2}.

\subsection[Proof of Theorem 2.6]{Proof of Theorem \texorpdfstring{ \ref{t5}}{2.6}}
In this section we prove Theorem \ref{t5}. We start proving that the sequence $\{\Lambda^{n,f}_t; t\in [0,T]\}_{\varepsilon >0}$ is tight. By Mitoma's criterion (stated in Proposition \ref{Mitoma}), it is enough to show that $\{\Lambda_t^{n,f}(u); t\in [0,T]\}_{\varepsilon >0}$, for $u\in{\mc S(\bb R)}$ is tight. Recall that $\Var(\psi_f(\ell);\nu_\rho) \leq c /\ell^2$. By the Cauchy-Schwarz inequality,
\[
 \bb E_\rho\Big[\Big(\int_0^ t\sum_{x \in \bb Z} \psi_f(\ell;\eta_s) h_xds\Big)^2\Big] \leq t^2 \int \Big(\sum_{x \in \bb Z} \psi_f(\ell) h_x\Big)^2 d\nu_\rho \leq \frac{ct^2}{\ell} \sum_{x \in \bb Z} h_x^2.
\]
This estimate combined with the energy estimate \eqref{energy} gives
\[
 \bb E_\rho\big[\Lambda_t^{n,f}(u)^2\big] \leq c\Big\{\frac{t\ell}{n} + \frac{t^2 n}{\ell}\Big\} \frac{1}{n} \sum_{x \in \bb Z} u(x/n)^2.
\]
The sum $n^{-1}\sum_{x \in \bb Z} u(x/n)^2$ converges to $\|u\|^2$, so it is bounded in $n$. Choosing $\ell = n\sqrt t$, we show that
\begin{equation}
\label{ec44}
 \bb E_\rho\big[\Lambda_t^{n,f}(u)^2\big] \leq c(\rho,f,u) t^{3/2}
\end{equation}
as long as $ t> \ell_0^2/n^2$. When $t < \ell_0^2/n^2$ we use the bound
\[
\begin{split}
 \bb E_\rho\big[\Lambda_t^{n,f}(u)^2\big]
    &\leq t^2 \int \Big(\sum_{x \in \bb Z} \big(\tau_x f -\varphi_f(\rho)\big)\,u(x/n)\Big)^2 d\nu_\rho\\
    &\leq c t^2 \sum_{x \in \bb Z} u(x/n)^2
    \leq c t^{3/2} \frac{1}{n} \sum_{x \in \bb Z} u(x/n)^2,
\end{split}
\]
which extends \eqref{ec44} to any $t \in [0,T]$. Once more, stationarity of $\eta_t$ and Kolmogorov-Centsov's compactness criterion show tightness of the sequence $\{\Lambda_t^{n,f}(u);t \in [0,T]\}_{\varepsilon >0}$, for $u\in{\mc S(\bb R)}$. We conclude that the sequence $\{ \Lambda_t^{n,f};t \in [0,T]\}_{\varepsilon>0}$ is tight in $\mc C([0,T];\mc S'(\bb R)\}$.

Let $\{\Lambda_t^f;t \in [0,T]\}$ be a limit point of $\{\Lambda_t^{n,f};t \in [0,T]\}_{\varepsilon >0}$. Using the second-order Boltzmann-Gibbs principle at time $tn^2$ with $\ell = \varepsilon n$, we see that
\[
 \bb E_\rho\Big[\Big(\Lambda_t^{n,f}(u)-\frac{\varphi_f''(\rho)}{2}\mc A_t^{n,\varepsilon}(u)\Big)^2\Big] \leq c\Big\{ \varepsilon t+ \frac{t^2}{\varepsilon^2 n}\Big\} \frac{1}{n} \sum_{x \in \bb Z} u(x/n)^2.
\]
Taking the limit $n \to \infty$, we obtain that
\[
 \bb E_\rho \Big[\Big(\
 \Lambda_t^f(u) -\frac{\varphi_f''(\rho)}{2} \mc A_t^\varepsilon(u)\Big)^2\Big] \leq c \varepsilon t \|u\|^2.
\]
Taking now the limit $\varepsilon \to 0$, we conclude that $\Lambda_t^f$ has the same finite-dimensional distributions of the process $\frac{\varphi_f''(\rho)}{2} \mc A_t$ and therefore they are equal in distribution. This finishes the proof of Theorem \ref{t5}.

\section{Additive functionals of the exclusion process}
\label{s5}
In this section we explain how to extend the proofs of the main theorems for two type of non-reversible processes: the mean-zero exclusion process and the weakly asymmetric exclusion process.

\subsection{The mean-zero exclusion process}
\label{s5.1}
The mean-zero, non-symmetric exclusion process $\eta_t$ introduced in Section \ref{s1.2} is perhaps the simplest example of a diffusive, non-reversible, non-gradient system. The adjective ``diffusive'' comes from the fact that the evolution of the density of particles for this model is non-trivial under a diffusive scaling of time. As discussed in Section \ref{s1.2}, the invariant measures are non-reversible for $\eta_t^{\mathrm{ex}}$ since we are assuming that $p(\cdot)$ is not symmetric. Let us define the density fluctuation field $\{\mc Y_t^n;t \in [0,T]\}$ like in Section \ref{s1.4}. The following result is not stated anywhere, but it can be proved following the methods in \cite{CLO} and using the formula for the diffusion coefficient $D(\rho)$ given in \cite{Sue}:

\begin{proposition}[\cite{Cha2,Sel}]
\label{mzep}
 The process $\{\mc Y_t^n;t \in [0,T]\}$ converges in distribution with respect to the $J_1$-Skorohod topology of $\mc D([0,T],\mc S'(\bb R))$ to the stationary solution of the infinite-dimensional Ornstein-Uhlenbeck equation
\begin{equation*}
 d\mc Y_t = D(\rho) \Delta \mc Y_t dt + \sqrt{2D(\rho) \chi(\rho)} \nabla d \mc M_t.
\end{equation*}
\end{proposition}

Notice that the definition of the $H_{-1}$-norm stated in \eqref{h-1} is insensitive to the asymmetric part of the dynamics. More precisely, we have the identity
\[
 \|f\|_{-1}^2 =\sup_{g \text{ local}} \big\{2\<f,g\>_\rho - \<g,-Sg\>_\rho\big\},
\]
where $S= (L+L^*)/2$ is the symmetric part of the generator $L$. Due to the product structure of the invariant measure $\nu_\rho$, the operator $S$ is easy to compute. In fact, $S$ corresponds to the generator of a exclusion process associated to the probability measure $p^s:\bb Z \setminus \{0\}  \to [0,1]$ defined as $p^s(z) = (p(z)+p(-z))/2$. The spectral gap inequality stated in Proposition \ref{p3} can be easily proved using a comparison argument and the irreducibility of $p(\cdot)$. Since the proof of Theorems \ref{t3}, \ref{t4} and \ref{t5} only rely on Kipnis-Varadhan's inequality and the convergence of the density fluctuation field to the corresponding Ornstein-Uhlenbeck process, those theorems generalize to the mean-zero exclusion process straightforwardly.

\subsection{Weakly asymmetric exclusion process}
\label{s5.2}
The weakly asymmetric exclusion process is not a process, but rather a sequence of processes indexed by the scaling parameter $n$. Let $\{a_n; n \in \bb N\}$ be a sequence of numbers converging to 0 as $n \to \infty$. For each $n \in \bb N$ large enough, let us define the probability measure $p_n(\cdot)$ as
\[
 p_n(z) =
\begin{cases}
 \frac{1}{2}(1+a_n), & x=1\\
\frac{1}{2}(1-a_n), & x=-1\\
0,& x \neq \pm 1.
\end{cases}
\]

Let us consider the exclusion process $\{\eta_{t}^{\mathrm{ex}}; t \geq 0\}$ associated to $p_n(\cdot)$. Although we will not write it explicitly, the process $\{\eta_{t}^{\mathrm{ex}}; t \geq 0\}$ is actually a sequence of processes indexed by $n$. Of course, if $a_n \equiv 0$, the process $\{\eta_t^{\mathrm{ex}}; t \geq 0\}$ is just the lattice gas process defined in Section \ref{s1.1} with $r \equiv 1/2$ and therefore Theorems \ref{t3}, \ref{t4}, \ref{t5} apply to this process. It is natural to ask under which  conditions on $\{a_n\}_{n \in \bb N}$ Theorems \ref{t3}, \ref{t4} and \ref{t5} are still valid for the sequence of processes $\{\eta_t^{\mathrm{ex}}; t \in [0,T]\}$. It turns out that there are two regimes under which this question can be answered in a satisfactory way.

The first regime is the following. Let us assume that there is a constant $a \in \bb R$ such that $\lim_{n \to \infty} n a_n =a$. In order to describe the scaling limit of the density field in this situation, we need to introduce the infinite-dimensional Ornstein-Uhlenbeck process with drift. Recall the notations introduced in Section \ref{s1.3}. For $D,\sigma >0$ and $v \in \bb R$ we say that a process $\{\mc Y_t; t \in [0,T]\}$ is a solution of the infinite-dimensional Ornstein-Uhlenbeck equation
\[
 d \mc Y_t = D \Delta \mc Y_t dt + a\nabla \mc Y_t dt + \sigma \nabla d\mc M_t
\]
if for any smooth trajectory $t \mapsto u_t$ from $[0,T] $ to $\mc S(\bb R)$, the process
\[
 \mc Y_t(u_t) -\mc Y_0(u_0) -\int_0^t \mc Y_s((\partial_s + D \Delta -a\nabla)u_s) ds
\]
is a martingale of quadratic variation $\sigma^2 \int_0^t \|\nabla u_s\|^2 ds$.
The following proposition has been proved in \cite{Cha}, \cite{D-MPS} and \cite{DG}:

\begin{proposition}
 \label{WASEP1}
Fix $\rho \in (0,1)$. Suppose that $\eta_0^{\mathrm{ex}}$ is distributed according to $\nu_\rho$. Let $\{\mc Y_t^n; t \in [0,T]\}$ be the density field defined on $u \in \mc S(\bb R)$ as
\begin{equation}
\label{densfi}
 \mc Y_t^n(u) = \frac{1}{\sqrt n} \sum_{x \in \bb Z} \big(\eta_{tn^2}^{\mathrm{ex}}(x)-\rho\big) u(x/n).
\end{equation}
The process $\{\mc Y_t^n; t \in [0,T]\}$ converges in distribution with respect to the $J_1$-Skorohod topology of $\mc D([0,T],\mc S'(\bb R))$ to the stationary solution of the infinite-dimensional Ornstein-Uhlenbeck equation
\begin{equation}
\label{ec5.1}
 d\mc Y_t = \frac{1}{2} \Delta \mc Y_t dt + a(1-2\rho) \nabla \mc Y_t dt + \sqrt{\chi(\rho)} \nabla d \mc M_t.
\end{equation}
\end{proposition}
When $a =0$ or $\rho =1/2$, this proposition is just saying that the scaling limit of the density field is not affected by a small asymmetry (that is, when $n a_n \to 0$ as $n \to \infty$, or when $n a_n$ is bounded if $\rho = 1/2$). When $a \neq 0$, we see that $1/n$ is the exact order of magnitude of the asymmetry for which it appears a non-trivial modification of the scaling limit. Notice that independently of the value of $a_n$,
\[
 p_n^s(z) =
\begin{cases}
 \frac{1}{2}, & x = \pm 1\\
0, & x \neq \pm 1.
\end{cases}
\]
In particular, following the arguments presented in Section \ref{s5.1}, Theorems \ref{t3}, \ref{t4} and \ref{t5} and suitable modifications of Theorems \ref{t1}, \ref{t2} hold true for the weakly asymmetric exclusion process and for the Ornstein-Uhlenbeck process with drift. Here we just state Theorems \ref{t1} and \ref{t4} for the infinite-dimensional Ornstein-Uhlenbeck process with drift and the weakly asymmetric exclusion process.

\begin{theorem}
 \label{twasep1}
Let $\{\mc Y_t; t \in [0,T]\}$ be the stationary solution of \eqref{ec5.1}. For $\varepsilon \in (0,1)$ and $t \in [0,T]$, define $\mc Z_t^\varepsilon$ as
\[
\mc Z_t^ \varepsilon = \int_0^t \mc Y_s(i_\varepsilon) ds.
\]
Then the sequence of processes $\{\mc Z_t^\varepsilon; t \in [0,T]\}_{\varepsilon}$ converges in distribution with respect to the uniform topology of $\mc C([0,T]; \bb R)$ as $\varepsilon \to 0$, to a Gaussian process $\{\mc Z_t;t \in [0,T]\}$ of stationary increments, satisfying
\[
 \bb E\big[\mc Z_t^2 \big] =\chi(\rho)\sqrt{\frac{2}{\pi}} \int_0^t \frac{(t-s)e^{-(a(1-2\rho))^2s/2}}{\sqrt s}ds.
\]
\end{theorem}

\begin{theorem}
 \label{twasep2}
Fix $\rho \in (0,1)$. Let us assume that $\eta_0^{\mathrm{ex}}$ is distributed according to $\nu_\rho$. The process $\{\Gamma_t^n; t \in [0,T]\}$ defined as
\[
 \Gamma_t^n = \frac{1}{n^{3/2}} \int_0^{tn^2} (\eta_s^{\mathrm{ex}}(0)-\rho) ds
\]
converges in distribution with respect to the uniform topology in $\mc C([0,T], \bb R)$ to the process $\{\mc Z_t; t \in [0,T]\}$ defined in Theorem \ref{twasep1}.
\end{theorem}

\subsection{Weakly asymmetric exclusion process and KPZ equation}

In the previous section, we showed how to obtain the scaling limit of additive functionals of a weakly asymmetric exclusion process. Notice that when $\rho =1/2$, the processes appearing in Proposition \ref{WASEP1} and Theorems \ref{twasep1}, \ref{twasep2} are the same appearing in the symmetric case. In other words, the introduction of a weak asymmetry does not change the scaling limits of the density of particles and of the occupation time. Therefore, it makes sense to ask whether a slower decay of the asymmetry (that is, a {\em stronger} weak asymmetry) will change the limiting processes. It turns out that for density $\rho =1/2$, the relevant scaling of the asymmetry is $1/\sqrt n$ instead of $1/n$. Let us assume in this section that $\rho =1/2$ and that there is a constant $a \in \bb R$ such that $\sqrt n a_n \to a$ as $n \to \infty$. The following theorem has been proved in \cite{BG} (see also \cite{ACQ}, \cite{GJ}):

\begin{proposition}
 \label{p5.2}
Fix $\rho =1/2$. Suppose that $\eta_0^n$ is distributed according to $\nu_\rho$ and suppose that $\lim_{n \to \infty} \sqrt{n} a_n =a \in \bb R$. Define the density field $\{\mc Y_t^n; t \in [0,1]\}$ as in \eqref{densfi}. Then the process $\{\mc Y_t^n; t \in [0,1]\}$ converges in distribution with respect to the $J_1$-Skorohod topology of $\mc D([0,T];\mc S'(\bb R))$ to the stationary Hopf-Cole solution of the stochastic Burgers equation
\begin{equation}
 \label{kpz}
 d \mc Y_t = \frac{1}{2}\Delta \mc Y_t dt + a\big(\nabla \mc Y_t\big)^2 dt + \sqrt{\chi(\rho)} \nabla d\mc M_t.
\end{equation}
\end{proposition}

We refer to \cite{BG} for a detailed discussion about Hopf-Cole solutions of \eqref{kpz}. Equation \eqref{kpz} corresponds to the formal gradient of the so-called {\em KPZ equation}. We refer to \cite{GJ} for a more detailed discussion. Of course we use the additional adjective ``stationary'' meaning that the fixed-time distributions of $\{\mc Y_t; t \in [0,T]\}$ are independent of $t$.

For the stationary solution $\{\mc Y_t; t \in [0,T]\}$ of the stochastic Burgers equation \eqref{kpz}, Theorems \ref{t2} and \ref{t5} were proved in \cite{GJ}. Notice that the proof of Theorem \ref{t1} is robust in the sense that it only requires the local Boltzmann-Gibbs principle and the convergence of the density field. Therefore, Theorem \ref{t1} also holds for the stationary solution of the stochastic Burgers equation \eqref{kpz}. This fact is a novelty in the literature, and we state it as a theorem:

\begin{theorem}
 \label{t1forkpz}
Let $\{\mc Y_t; t \in [0,T]\}$ be the stationary Hopf-Cole solution of the stochastic Burgers equation \eqref{kpz}. For $\varepsilon >0$, let us define
\[
 \mc Z_t^ \varepsilon = \int_0^ t \mc Y_s(i_\varepsilon) ds.
\]
Then there exists a real-valued process $\{\mc Z_t; t \in [0,T]\}$ such that $\{\mc Z_t^\varepsilon; t \in [0,T]\}$ converges in distribution with respect to the uniform topology of $\mc C([0,T]; \bb R)$ to $\{\mc Z_t; t \in [0,T]\}$.
\end{theorem}

The analog of Theorem \ref{t4} also holds for the process $\{\mc Y_t;t \in [0,T]\}$:

\begin{theorem}
 \label{t4forkpz}
Let $f: \Omega \to \bb R$ a local function with $\varphi_f(1/2) =0$. The process $\{\Gamma_t^n(f); t \in [0,T]\}$ defined as
\[
 \Gamma_t^n(f) = \frac{1}{n^{3/2}} \int_0^{tn^2} f(\eta_s) ds
\]
converges in distribution with respect to the uniform topology of $\mc C([0,T];\bb R)$ to the process $\{ \varphi_f'(1/2)\mc Z_t; t \in [0,T]\}$, where $\mc Z_t$ is the process defined in Theorem \ref{t1forkpz}.
\end{theorem}

\begin{remark}
In this section we chose $\rho =1/2$ because in this case the drift of the Ornstein-Uhlenbeck process appearing in Proposition \ref{WASEP1} vanishes. When the drift does not vanish, we can perform a Galilean transformation of the density field in such a way that, under the new reference system, the fluctuations of the density are governed by a driftless Ornstein-Uhlenbeck process. Under this new system of reference, Proposition \ref{p5.2} and Theorems \ref{t1forkpz}, \ref{t4forkpz} also hold (see Section 2.2 of \cite{GJ}).
\end{remark}

\section*{Acknowledgements}
P.G. thanks the hospitality of Universit\'e Paris-Dauphine, where this work was initiated; IMPA and Courant Institute of Mathematical Sciences, where this work was finished. P.G. thanks FCT for the research project PTDC/MAT/109844/2009: "Non-Equilibrium Statistical Physics".  PG thanks the Research Centre of Mathematics of the University of Minho, for the financial support provided by "FEDER" through the "Programa Operacional Factores de Competitividade – COMPETE" and by FCT through the research project PEst-C/MAT/UI0013/2011.

M.J. would like to thank the warm hospitality of the Fields Institute.

We are grateful to Capes and FCT for the reseach project: "Non-Equilibrium Statistical Mechanics of Stochastic Lattice Systems" with reference FCT 291/11.

\end{document}